\documentclass[12pt]{amsart}
\usepackage{amssymb,latexsym,pstricks,amsmath,amsthm,amsfonts, enumerate}
\usepackage{color}
\usepackage{auto-pst-pdf}
\usepackage{tikz}
\usepackage{pifont}

\numberwithin{equation}{section}

\newtheorem{theorem}{Theorem}[section]
\newtheorem{proposition}[theorem]{Proposition}

\newtheorem{corollary}[theorem]{Corollary}
\newtheorem{lemma}[theorem]{Lemma}

\theoremstyle{definition}

\newtheorem{definition}[theorem]{Definition}

\let\oldmarginpar\marginpar
\renewcommand\marginpar[1]{\-\oldmarginpar[\raggedleft\small\sf
#1]{\raggedright\small\sf #1}}

\newcommand{\myAA}{\mathcal{A}}

\newcommand{\FF}{\mathcal{F}}

\newcommand{\ZZ}{\mathbb{Z}}

\newcommand{\QQ}{\mathbb{Q}}

\newcommand{\mat}[4]{\left(\!\!\begin{array}{cc}
#1 & #2 \\ #3 & #4 \\
\end{array}\!\!\right)}
\newcommand{\rem}[1]{\left\langle#1\right\rangle}
\newcommand{\myx}{x_{b\leftrightarrow c}}
\newcommand{\myPS}{PS_{b\leftrightarrow c}}

\begin{document}

\title[Positivity and tameness in rank 2 cluster algebras]
{Positivity and tameness in rank 2 cluster algebras}

\author{Kyungyong Lee}
\address{\noindent Department of Mathematics, Wayne State University, Detroit, MI 48202, USA}
\email{klee@math.wayne.edu}

\author{Li Li}
\address{\noindent Department of Mathematics and Statistics, Oakland University, Rochester, MI 48309, USA}
\email{li2345@oakland.edu}

\author{Andrei Zelevinsky}
\address{\noindent Department of Mathematics, Northeastern University,
Boston, MA 02115, USA}
\email{andrei@neu.edu}


\date{March 22, 2013}

\thanks{Research 
supported in part by NSF grants DMS-0901367 (K.~L.) and DMS-1103813 (A.~Z.)}

\begin{abstract}
We study the relationship between the positivity property in a rank 2 cluster algebra, and the property of such an algebra to be tame.
More precisely, we show that a rank 2 cluster
 algebra has a basis of indecomposable positive elements if and only if it
 is of finite or affine type. This statement 
disagrees with a conjecture by Fock and Goncharov.
 \end{abstract}

\maketitle

\bigskip


\section{Introduction and main results}
\label{sec:intro}

This note continues the study of the positivity structure for coefficient-free rank~$2$ cluster algebras initiated in \cite{sz,llz}. 
These algebras can be quickly defined as follows: they form a $2$-parametric family depending on a pair of positive integers~$(b,c)$, and the cluster algebra $\myAA(b,c)$ is the
subring of the \emph{ambient field} $\FF = \QQ(x_1,x_2)$ generated by the \emph{cluster variables} $x_m$ for all $m \in \ZZ$, where the (two-sided) sequence of cluster variables is given recursively by the relations
\begin{equation}
\label{eq:clusterrelations}
x_{m-1} x_{m+1} = \left\{
\begin{array}[h]{ll}
x_m^b + 1 & \quad \mbox{for  $m$ odd;} \\
x_m^c + 1 & \quad \mbox{for $m$ even.}
\end{array}
\right.
\end{equation} 
Recall also that the sets $\{x_m,x_{m+1}\}$ for
$m\in\ZZ$ are called \emph{clusters}, and the ambient field $\FF$ is naturally identified with $\QQ(x_m,x_{m+1})$ for any $m \in \ZZ$.

Despite such an elementary definition, we find the structure theory of these algebras rather deep, and some of the natural questions surprisingly difficult. 
Here is a fundamental result, which is a special case of the Laurent phenomenon discovered and proved in \cite{fz-ClusterI,fz-Laurent,fz-ClusterIII}: every cluster variable is not just a rational function in the two elements of any given cluster but a Laurent polynomials with integer coefficients. The
following stronger result is a special case of the results
in~\cite{fz-ClusterIII}:
\begin{equation}
\label{eq:upper} 
\myAA = \bigcap_{m \in \ZZ} \ZZ[x_m^{\pm 1},
x_{m+1}^{\pm 1}],
\end{equation}
where $\ZZ[x_m^{\pm 1}, x_{m+1}^{\pm 1}]$ denotes the ring of
Laurent polynomials with integer coefficients in $x_m$ and
$x_{m+1}$.

The  main focus in our study of the algebras $\myAA(b,c)$ is \emph{positivity}. Recall that a non-zero element $x\in\myAA(b,c)$ is
\emph{positive at a cluster $\{x_m,x_{m+1}\}$} if all the coefficients in
the expansion of $x$ as a Laurent polynomial in $x_m$ and
$x_{m+1}$ are positive.
We say that $x \in\myAA(b,c)$ is \emph{positive} if it is positive at all the clusters. 
 
Following \cite{sz}, we introduce the following 
important definition. 

\begin{definition}
\label{de:positive-indecomposables} 
A positive element $x\in\myAA(b,c)$ is
\emph{indecomposable} if it cannot be expressed as the sum of two positive elements. 
\end{definition}

Recall that $\myAA(b,c)$ is of \emph{finite} (resp. \emph{affine}) type if $bc \leq 3$ (resp. $bc = 4$). 
It is also common to refer to the case $bc \leq 4$ as \emph{tame}, and that of $bc > 4$ as \emph{wild} (this terminology comes from the theory of quiver representations). 
One of the main results of \cite{sz} is the following: if $\myAA(b,c)$ is tame then indecomposable positive elements form 
a $\ZZ$-basis in $\myAA(b,c)$.

Motivated by this result (among other considerations) it was conjectured in \cite[Conjecture~5.1]{fg} that indecomposable positive elements form  a $\ZZ$-basis in any cluster algebra. However, already the authors of \cite{sz} suspected that this is not true, and this suspicion was detailed and stated explicitly in \cite{llz}. Here we finally settle this question by proving the following. 

\begin{theorem}
\label{th:positivity-tameness}
The set of indecomposable positive elements forms a $\ZZ$-basis for $\mathcal{A}(b,c)$ if and only if $bc \leq 4$ i.e. $\mathcal{A}(b,c)$ is tame.
\end{theorem}

Recall that in \cite{llz} for each $(b,c)$ there was introduced a family $\{x[a_1,a_2] \ : \ (a_1,a_2) \in \ZZ^2\}$ of \emph{greedy elements} in $\myAA(b,c)$, and it was proved among other things that they are indecomposable positive, and that they form a $\ZZ$-basis in $\myAA(b,c)$. Taking this into account, it is easy to see that the following conditions on $(b,c)$ are equivalent:

\begin{enumerate}
\item Indecomposable positive elements do not form a $\ZZ$-basis for $\mathcal{A}(b,c)$. 

\item Indecomposable positive elements in $\myAA(b,c)$ are linearly dependent.

\item There exists a non-greedy indecomposable positive element in $\myAA(b,c)$.

\item  There exists a positive element $p \in \myAA(b,c)$ whose expansion in the basis of greedy elements has at least one negative coefficient. 
\end{enumerate} 

For instance, to deduce (3) from (4) take any expansion of $p$ into the sum of indecomposable positive elements, and note that it is \emph{different} from the expansion of $p$ in the basis of greedy elements. Hence at least one of indecomposable positive components of $p$ must be non-greedy. 

We see that to prove Theorem~\ref{th:positivity-tameness} it suffices to show that in the wild case there always exists an element~$p$ satisfying (4). We exhibit such an element explicitly as follows. 

\begin{theorem} 
\label{th:p-explicit}
Suppose that $bc > 4$, i.e., $\myAA(b,c)$ is wild. Define an element
$p \in \myAA(b,c)$ as follows:
\begin{equation}
\label{eq:p-explicit}
p = \left\{
\begin{array}[h]{ll}
x[bc-b+1,c+1]+x[b+1,bc-c+1]-x[1,1] & \, \mbox{if $\min(b,c)>1$;} \\
x[b+2,3] + x[b+2, b-1] - x[2,1] & \, \mbox{if $c=1$;} \\
x[3,c+2] + x[c-1, c+2] - x[1,2] & \, \mbox{if $b=1$.}
\end{array}
\right.
\end{equation} 
Then $p$ is positive, hence satisfies condition (4) above.
\end{theorem}

To show that the element $p$ given by \eqref{eq:p-explicit} is positive, we use the group of automorphisms $W$ of $\myAA(b,c)$ introduced in \cite{llz}. By the definition, $W$ is generated by the involutions $\sigma_\ell$ for $\ell \in \ZZ$, where $\sigma_\ell$ acts on cluster variables by a permutation $\sigma_\ell(x_m) = x_{2 \ell-m}$.  It is easy to see that $W$ is a dihedral group generated by $\sigma_1$ and $\sigma_2$ (this group is finite if $\myAA$ is of finite type, and infinite otherwise). As shown in \cite[Proposition~1.8]{llz}, the set of greedy elements is $W$-invariant, and  the automorphisms $\sigma_1$ and
$\sigma_2$ act on greedy elements as follows:
\begin{equation}
\label{eq:sigma-Q} 
\sigma_1(x[a_1, a_2]) = x[a_1, c [a_1]_+ - a_2], \quad 
\sigma_2 (x[a_1, a_2]) = x[b [a_2]_+ - a_1, a_2] 
\end{equation}
for all $(a_1, a_2) \in \ZZ^2$, where we use the standard notation $[a]_+ =  \max (a, 0)$.

Clearly, $W$ acts transitively on the set of all clusters of $\myAA(b,c)$. 
Thus the positivity of~$p$ is equivalent to the property that $\sigma(p)$
is positive at the initial  cluster $\{x_1,x_2\}$ for every $\sigma \in W$. 

We identify $\ZZ^2$ with the root lattice associated with the (generalized) Cartan matrix
\begin{equation}
\label{eq:cartan-rank2}
A=A(b,c)=\mat{2}{-b}{-c}{2} 
\end{equation} 
using an unorthodox convention that the simple roots $\alpha_1$ and $\alpha_2$ are identified with $(0,1)$ and $(1,0)$ respectively.  
The \emph{Weyl group} $W(A)$ is a group of linear
transformations of $\mathbb{Z}^2$ generated by two \emph{simple reflections} $s_1$ and $s_2$
whose action in $\ZZ^2$ is given by
\begin{equation}
\label{eq:s1-s2}
s_1 = \mat{1}{0}{c}{-1} \ ,\quad
s_2 = \mat{-1}{b}{0}{1}\ .
\end{equation}
Comparing this with \eqref{eq:sigma-Q} we see that $s_1$ agrees with $\sigma_1$, and $s_2$ agrees with $\sigma_2$ on $\ZZ_{\geq 0}^2$.

It is well-known (see e.g., \cite{kac}) that the difference between tame and wild cases manifests itself in the appearance and behavior of \emph{imaginary roots}. 
According to \cite{kac}, the set $\Phi^{\rm im}_+$ of \emph{positive imaginary roots} can be defined as follows:
\begin{equation}
\label{eq:imaginary-roots}
\Phi^{\rm im}_+ = \{(a_1,a_2) \in (\ZZ_{> 0})^2 \, : \,
Q(a_1, a_2) \leq 0\} \,
\end{equation}   
where $Q$ is the $W(A)$-invariant quadratic form on $\ZZ^2$  given by
\begin{equation}
\label{eq:scalar-product}
Q(a_1, a_2) =
c a_1^2 - bc a_1 a_2 + b a_2^2 \ .
\end{equation}
It is known (and easy to check) that:
\begin{itemize} 
\item in the finite type $bc < 4$ the form $Q$ is positive definite, and so $\Phi^{\rm im}_+ = \emptyset$; 
\item in the affine type $bc = 4$, we have $Q(a_1, a_2) \geq 0$ on $\ZZ^2$, and  $\Phi^{\rm im}_+ = \{(a_1,a_2) \in \ZZ_{> 0} \, : \,
Q(a_1, a_2)) = 0\}$ is the set of integer positive multiples of the \emph{minimal} imaginary root;
\item in the wild case $bc > 4$, the form $Q$ does not vanish on $\ZZ^2$, and we have 
$$\aligned
\label{eq:imaginary-inequalities}
\Phi^{\rm im}_+ & = \{(a_1,a_2) \in (\ZZ_{> 0})^2 \, : \,
Q(a_1, a_2) < 0\}\\
& = 
 \{(a_1,a_2) \in (\ZZ_{> 0})^2 \, : \,
\frac{bc - \sqrt{bc(bc-4)}}{2b} < \frac{a_2}{a_1} <
\frac{bc + \sqrt{bc(bc-4)}}{2b}\} \ .
\endaligned$$
\end{itemize}   

In this note our main interest is in the wild case. 
In this case we can and will identify the group $W$ of automorphisms of $\myAA(b,c)$ generated by $\sigma_1$ and $\sigma_2$ with the Weyl group $W(A)$ by identifying $\sigma_1$ with $s_1$, and $\sigma_2$ with $s_2$. 
The above facts imply that $\Phi^{\rm im}_+$ is $W(A)$-invariant, and we have the following useful property.

\begin{proposition} 
\label{pr:sigma-w}
In the wild case $bc > 4$, if $\sigma \in W$ and $w \in W(A)$ are identified with each other, and $(a_1,a_2) \in \Phi^{\rm im}_+$ then we have $\sigma(x[a_1,a_2]) = x[w(a_1,a_2)]$. 
\end{proposition}

Returning to Theorem~\ref{th:p-explicit}, it is easy to check that all the elements $x[a_1,a_2]$ appearing in the right hand side of \eqref{eq:p-explicit} correspond to positive imaginary roots (see Lemma \ref{lem:imaginary-components}). 
Thus to prove Theorems~\ref{th:positivity-tameness} and 
~\ref{th:p-explicit} it suffices to establish the following key lemma.

\begin{lemma}
\label{lem:key-lemma}
In the setup of Theorem~\ref{th:p-explicit}, for every $w \in W$, the element obtained from $p$ by replacing each element $x[a_1,a_2]$ in the right hand side of \eqref{eq:p-explicit} with $x[w(a_1,a_2)]$
 is positive at the initial  cluster $\{x_1,x_2\}$.
 \end{lemma}
 
 The proof of Lemma~\ref{lem:key-lemma} is carried out in Section~\ref{sec:main-proofs}.
 All the tools in our proof  are already developed in \cite{llz}. 
 The two most important ingredients are as follows:
 \begin{itemize}
 \item  A combinatorial definition of greedy elements in terms of \emph{compatible pairs} (to be recalled later). 
The combinatorics of compatible pairs plays a crucial part in our argument.
\item A precise description of the set of Laurent monomials that appear in the expansion of $x[a_1,a_2]$ in terms of the initial cluster
$\{x_1,x_2\}$, for an arbitrary positive imaginary root $(a_1,a_2)$. 
Specifically, we show that the upper bound for this set given in \cite[Proposition~4.1, Case~6]{llz} is exact.
\end{itemize}
\section{Compatibility and greedy elements in rank 2 cluster algebras}
In this section we recall some definitions and results from \cite{llz}.

Let $(a_1, a_2)$ be a pair of nonnegative integers. 
A \emph{Dyck path} of type $a_1\times a_2$ is a lattice path
from $(0, 0)$  to $(a_1,a_2)$ that 
never goes above the main diagonal joining $(0,0)$ and $(a_1,a_2)$.
Among the Dyck paths of a given type $a_1\times a_2$, there is a (unique) \emph{maximal} one denoted by 
$\mathcal{D} = \mathcal{D}^{a_1\times a_2}$. 
It is defined by the property that any lattice point strictly above $\mathcal{D}$ is also strictly above the main diagonal.

Let $\mathcal{D}=\mathcal{D}^{a_1\times a_2}$.  Let $\mathcal{D}_1=\{u_1,\dots,u_{a_1}\}$ be the set of horizontal edges of $\mathcal{D}$ indexed from left to right, and $\mathcal{D}_2=\{v_1,\dots, v_{a_2}\}$ the set of vertical edges of $\mathcal{D}$ indexed from bottom to top.  
Given any points $A$ and $B$ on $\mathcal{D}$, let $AB$ be the subpath starting from $A$, and going in the Northeast direction until it reaches $B$ (if we reach $(a_1,a_2)$ first, we continue from $(0,0)$). By convention, if $A=B$, then $AA$ is the subpath that starts from $A$, then passes $(a_1,a_2)$ and ends at $A$. If we represent a subpath of $\mathcal{D}$ by its set of edges, then for $A=(i,j)$ and $B=(i',j')$, we have
$$
AB= 
\begin{cases}
\{u_k, v_\ell: i < k \leq i', j < \ell \leq j'\}, \quad\textrm{if $B$ is to the Northeast of $A$};\\
\mathcal{D} - \{u_k, v_\ell: i' < k \leq i, j' < \ell \leq j\}, \quad\textrm{otherwise}.
\end{cases}
$$
We denote by $(AB)_1$ the set of horizontal edges in $AB$, and by $(AB)_2$ the set of vertical edges in $AB$. 
Also let $AB^\circ$ denote the set of lattice points on the subpath $AB$ excluding the endpoints $A$ and $B$ (here $(0,0)$ and $(a_1,a_2)$ are regarded as the same point).

\begin{definition}\cite{llz}
\label{df:compatible}
For $S_1\subseteq \mathcal{D}_1$, $S_2\subseteq \mathcal{D}_2$, we say that the pair $(S_1,S_2)$ is compatible if for every $u\in S_1$ and $v\in S_2$, denoting by $E$ the left endpoint of $u$ and $F$ the upper endpoint of $v$, there exists a lattice point $A\in EF^\circ$ such that 
\begin{equation}
\label{0407df:comp}
|(AF)_1|=b|(AF)_2\cap S_2|\textrm{\; or\; }|(EA)_2|=c|(EA)_1\cap S_1|.\end{equation}
\end{definition}

One of the main results of \cite{llz} is the following combinatorial expression for greedy elements. 
\begin{theorem}\cite{llz}
\label{th:greedy-combinatorial}
For every $(a_1,a_2) \in \ZZ^2_{\geq 0}$, the greedy element $x[a_1,a_2] \in \myAA(b,c)$ at $(a_1,a_2)$ is given by
\begin{equation}
\label{eq:greedy-Dyck-expression}
x[a_1,a_2] = x_1^{-a_1}x_2^{-a_2}\sum_{(S_1,S_2)}x_1^{b|S_2|}x_2^{c|S_1|},
\end{equation}
where the sum is over all compatible pairs $(S_1,S_2)$ in $\mathcal{D}^{a_1\times a_2}$.
\end{theorem}

For the purposes of this paper we can view \eqref{eq:greedy-Dyck-expression} as a \emph{definition} of greedy elements.

\section{extremal pairs and their compatibility}
In this section we define extremal pairs and study their compatibility condition. As a consequence, we give a precise description of the set of Laurent monomials that appear in the expansion of $x[a_1,a_2]$ in terms of the initial cluster
$\{x_1,x_2\}$, for an arbitrary positive imaginary root $(a_1,a_2)$. 

\begin{definition}
Let $(a_1,a_2) \in \ZZ^2_{\geq 0}$, and 
$\mathcal{D} = \mathcal{D}^{a_1\times a_2}$. 
Let $s_1$, $s_2$ be two integers such that $0\le s_1\le a_1$, and $0\le s_2\le a_2$. If $S_1=\{u_1,\dots,u_{s_1}\}$ is the set of the first $s_1$ horizontal edges in $\mathcal{D}$, and $S_2=\{v_{a_2-s_2+1},\dots,v_{a_2}\}$ is the set of the last $s_2$ vertical edges in $\mathcal{D}$, then we call $(S_1,S_2)$ the \emph{extremal pair} (in $\mathcal{D}$) of size $(s_1; s_2)$. 
\end{definition}

As in  \cite{llz}, we denote by $c(p,q)$ the coefficients of $x[a_1,a_2]$: $$x[a_1,a_2]=x_1^{-a_1}x_2^{-a_2}\sum_{p,q\ge0}c(p,q)x_1^{bp}x_2^{cq},$$ and define the pointed support of $x[a_1,a_2]$ to be
$$PS[a_1,a_2]=\{(p,q) \in \ZZ_{\geq 0}^2 : c(p,q) \neq 0\} \ .$$
Let $P = P[a_1,a_2]\in\mathbb{R}^2$ be the region bounded by the broken line 
$$(0,0), \,\,(a_2,0), \,\,(a_1/b, a_2/c), \,\, (0,a_1), \,\, (0,0),$$
with the convention that $P$ includes sides $[(0,0), (a_2,0)]$ and $[(0,a_1),(0,0)]$ but excludes the rest of the boundary.
The following lemma gives a simple geometric characterization of positive imaginary roots.
 
\begin{lemma}\label{lemma: concave P} 
A lattice point $(a_1, a_2) \in \ZZ^2_{>0}$ is a positive imaginary root if and only if 
$a_1p+a_2q\le a_1a_2$ for every $(p,q)\in P$, in other words, $P$ is contained in the triangle with vertices $(0,0)$, $(a_2,0)$ and $(0,a_1)$.
\end{lemma}
\begin{proof}
The latter condition is equivalent to the condition that the vertex $(a_1/b, a_2/c)$  lies inside or on the boundary of the triangle, which is equivalent to the inequality $a_1 \cdot a_1/b + a_2 \cdot a_2/c \leq a_1 a_2$. This is in turn equivalent to $Q(a_1,a_2) \leq 0$, i.e., to the condition that  $(a_1,a_2)$ is a positive imaginary root.
\end{proof}


\begin{lemma}\label{lemma: S1 in front of S2}
Suppose $(a_1, a_2)$ is a positive imaginary root, and $(p, q)$ is a pair of positive integers. Let $(S_1, S_2)$ be the extremal pair of size $(q; p)$ in $\mathcal{D}$. Then the following conditions are equivalent:

(1) every $u \in S_1$ precedes every $v \in S_2$;

(2) $a_1 p + a_2 q < a_1 a_2 + a_1 + a_2$. 
\end{lemma}
\begin{proof} 
The condition (1) is equivalent to the condition that the edge $u_q$ is not higher than the lower endpoint of $v_{a_2-p+1}$, that is, $\lfloor(q-1)a_2/a_1\rfloor\le a_2-p$. This inequality is equivalent to $(q-1)a_2/a_1-1<a_2-p$, thus is equivalent to (2). 
\end{proof}

Recall that in \cite[Proposition 4.1 (6)]{llz} we showed that $PS[a_1,a_2]$ is contained in $P$.
Our goal in the rest of this section is to prove the following strengthening of this result.

\begin{proposition}
\label{prop:support}
Assume that $(a_1,a_2)$ is a positive imaginary root. If $(p,q)$ is a lattice point in $P$ and both $p$ and $q$ are positive, then the extremal pair $(S_1,S_2)$ of size $(q; p)$ is compatible. 
\end{proposition} 

\begin{corollary}
Assume that $(a_1,a_2)$ is a positive imaginary root. The set $PS[a_1,a_2]$ is the set of all lattice points in the region $P$. 
\end{corollary}
\begin{proof}
Because $PS[a_1,a_2]$ is contained in $P$, we only need to show that for every lattice point $(p,q)$ in $P$, the extremal pair $(S_1,S_2)$ of size $(q;p)$ is compatible. If $p=0$ or $q=0$, the conclusion is immediate. Other cases follow from Proposition \ref{prop:support}.
\end{proof}

The following lemma plays a key role in our proof of Proposition~\ref{prop:support}. 

\begin{lemma}\label{extremal pair compatible} Suppose $(a_1,a_2)$ is a positive imaginary root, and $(p,q)$ a pair of positive integers satisfying $a_1p+a_2q < a_1a_2+a_1+a_2$. Let $(S_1,S_2)$ be the extremal pair of size $(q;p)$ in $\mathcal{D}$.
Then $(S_1, S_2)$ is compatible if every horizontal edge $u \in S_1$  with the left endpoint $E$, and every vertical edge $v\in S_2$  with the top endpoint $F$ satisfy at least one of the following inequalities:
\begin{equation}\label{0407df:comp'1}
|(EF)_1|>b|(EF)_2\cap S_2|,
\end{equation}
\begin{equation}\label{0407df:comp'2}
|(EF)_2|>c|(EF)_1\cap S_1|,
\end{equation}
Equivalently,  the extremal pair $(S_1, S_2)$ is compatible if for every $1\le p'\le p$, $1\le q'\le q$, at least one of the following inequalities holds:
\begin{equation}\label{comp'1}
(ba_2-a_1)(p'-1)-a_2(q'-1)>(bp-a_1)a_2,
\end{equation}
\begin{equation}\label{comp'2} 
(ca_1-a_2)(q'-1)-a_1(p'-1)> (cq-a_2)a_1.
\end{equation}
\end{lemma}

\begin{proof}
Define $f(AB)=b|(AB)_2\cap S_2|-|(AB)_1|$. Without loss of generality, we assume \eqref{0407df:comp'1}, that is, $f(EF)<0$. As $A$ moves through the lattice points along $\mathcal{D}$ from $E$ to the lower endpoint of $v$, $f(AF)$ either decreases by $b$, stays constant, or increase by $1$ at each step; moreover, it starts with the negative value $f(EF)$ and end at the positive value $b$. Thus there exists $A\in EF^\circ$ such that $f(AF)=0$. Since the same argument works for every $u$ and $v$, we conclude that $(S_1,S_2)$ is compatible.

Next, we show that, for $u=u_{q'}$ and $v=v_{a_2-p'+1}$, \eqref{0407df:comp'1} is equivalent to \eqref{comp'1}. Indeed,
$$\aligned
\eqref{0407df:comp'1}&\Leftrightarrow (a_1-\lfloor a_1(p'-1)/a_2\rfloor)-(q'-1)> b(p-p'+1)\\
&\Leftrightarrow \lfloor a_1(p'-1)/a_2\rfloor\le a_1-q'-b(p-p'+1)\\
&\Leftrightarrow a_1(p'-1)/a_2-1< a_1-q'-b(p-p'+1)\Leftrightarrow\eqref{comp'1}.
\endaligned$$

The equivalence of \eqref{0407df:comp'2} and \eqref{comp'2} is proved similarly.
\end{proof}

\begin{proof}[Proof of Proposition \ref{prop:support}] 
It is then easy to check that $(p,q)$ lies in $P$ if and only if at least one of the following two conditions hold:
\begin{equation}\label{eq:cond1} 
a_1\ge bp \textrm{ and }(ca_1-a_2)(a_1-bp)> (cq-a_2)a_1,
\end{equation}
\begin{equation}\label{eq:cond2}
a_2\ge cq \textrm{ and }(ba_2-a_1)(a_2-cq)> (bp-a_1)a_2.
\end{equation}
Without loss of generality we assume \eqref{eq:cond1} holds.

Thanks to Lemma \ref{extremal pair compatible}, to show the compatibility of $(S_1,S_2)$, it suffices to prove
$$R\subset H_1\cup H_2,$$ where
$$\aligned
R&=\{(p',q')\in\mathbb{Z}^2|1\le p'\le p, 1\le q'\le q\},\\
H_1&=\{(p',q')\in \mathbb{R}^2|(ba_2-a_1)(p'-1)-a_2(q'-1)> (bp-a_1)a_2\},\\
H_2&=\{(p',q')\in \mathbb{R}^2|(ca_1-a_2)(q'-1)-a_1(p'-1)> (cq-a_2)a_1\}.
\endaligned
$$
Note that $H_1$ is the half plane below the line passing through $(1,a_1-bp+1)$ with slope $m_1=(ba_2-a_1)/a_2>0$, and $H_2$ is a half plane above a line 
with slope $m_2=a_1/(ca_1-a_2)>0$. Also note that $m_1> m_2$ because $Q(a_1,a_2)<0$. Thus to show that $R\subset H_1\cup H_2$, it suffices to show that $(1,a_1-bp+1)\in H_2$ (as illustrated in Figure \ref{H1H2}). But this is exactly the statement of \eqref{eq:cond1}.
\begin{figure}[h]
\begin{tikzpicture}[scale=.6]
\fill[red!50, fill opacity=0.5] (-.7,-1)--(3.2,8)--(8,8)--(8,-1); 
\fill[blue!50, fill opacity=0.5] (-1,1.5)--(8,5.5)--(8,8)--(-1,8); 
\draw[step=1,color=gray] (-1,-1) grid (8,8);
\draw[line width=1,color=black] (-1,0)--(8,0);
\draw[line width=1,color=black] (0,-1)--(0,8);
\foreach \x in {1,2,...,6}{
      \foreach \y in {1,2,...,6}{
        \node[draw,circle,inner sep=1pt,fill] at (\x,\y) {};
}}
\draw[line width=.3,dashed,color=blue] (-.7,-1)--(3.2,8);
\draw[line width=.3,dashed,color=blue] (-1,1.5)--(8,5.5);
\draw (0,-.5) node[anchor=west] {$H_1$};
\draw (0.2,2.7) node[anchor=east] {$H_2$};
\draw (4,3.5) node[anchor=east] {$R$};
\draw (-1.5,4) node[anchor=east] {\small $(1,a_1-bp+1)$};
\draw [->] (-1.5,4) -- (.8,3.1);
\draw (13,0) node[anchor=east] {\phantom{R}};
\end{tikzpicture}
\caption{}
\label{H1H2}
\end{figure}
\end{proof}

\section{Proofs of main resuls}
\label{sec:main-proofs}

In this section we prove Theorems~\ref{th:positivity-tameness} and ~\ref{th:p-explicit}. 
As discussed in Section~\ref{sec:intro}, it is enough to prove Lemma~\ref{lem:key-lemma}, which will be our goal.

Due to obvious symmetry, we can and will assume in the rest of this Section that 
\begin{equation}
\label{eq:symmetry}
\min(b,c) = c.
\end{equation}
In particular, we can disregard the last case in \eqref {eq:p-explicit}.
Now note that since  both $s_1$ and $s_2$ are involutions (see \eqref{eq:s1-s2}), each element of $W$
is one of the following (for $k\geq 0$):
$$w(1;k) = \underbrace{s_{\rem{k}}\cdots s_1 s_2 s_1  }_{k
\text{~factors}}, \,\,
w(2;k) = \underbrace{s_{\rem{k+1}}\cdots  s_2 s_1 s_2 }_{k
\text{~factors}} \ ;$$
here we use the convention $w(1;0) = w(2;0) = e$, the identity
element of $W$, and $\rem{k}=1$ if $k$ is odd, or $2$ if $k$ is even.

Let $$r_k = \left\{ \begin{array}{ll} b & \text{ for } k\text{ odd, } \\  c & \text{ for } k\text{ even.}    \end{array}     \right.$$

\begin{definition}\label{definition:alpha}
For $(a_1,a_2)\in\mathbb{Z}^2$, we define a sequence $(a(k))_{k\ge-1}$ with initial data $(a_1,a_2)$ as follows:
let $a(-1)=a_2$, $a(0) = a_1$,  and for $k>0$ recursively define 
\begin{equation}\label{a}
a(k)= r_{k-1} a(k-1) - a(k-2).
\end{equation}

We define sequences $(\alpha(k)),(\beta(k)),(\gamma(k))$ with initial data given in the table: 
\begin{center}
\begin{tabular}{|  l | c  |c | c|}
  \hline                  
 & $(\alpha(0),\alpha(-1))$&$(\beta(0),\beta(-1))$&$(\gamma(0),\gamma(-1))$\\
 \hline
  if $\min(b,c)>1$ & $(1,1)$ &$((c-1)b+1,c+1)$ &$(b+1,(b-1)c+1)$\\
\hline 
  if $c=1$ & $(2,1)$ & $(b+2,3)$ & $(b+2,b-1)$\\
 \hline  
\end{tabular}
\end{center}
\end{definition}

The next step is to show that all components of the element $p$ in \eqref{eq:p-explicit}
correspond to positive imaginary roots. 
Thus we show the following.

\begin{lemma}
\label{lem:imaginary-components}
Suppose we are in the wild case, that is, $bc > 4$. 
\begin{enumerate}
\item If $\min(b,c) =  c > 1$ then each of the vectors $(1,1)$, $(b+1,bc-c+1)$, and $(bc-b+1,c+1)$ is a positive imaginary root (recall that this means that the form $Q$ given by \eqref{eq:scalar-product} takes negative values at all these vectors). 
\item If $c=1$ and $b > 4$, then each of the vectors $(2,1)$, $(b+2,3)$, and $(b+2,b-1)$ is a positive imaginary root.
\end{enumerate} 
\end{lemma}
\begin{proof}
(1) A direct computation shows $Q(1,1)=b+c-bc$, $Q(b+1,bc-c+1)=Q(bc-b+1,c+1)=(2bc+1)(b+c-bc)$. Note that $b+c-bc = (2-c)c + (b-c)(1-c)$ is a sum of two non-positive terms. Furthermore, we see that the equality $b+c-bc = 0$ is achieved only when $b=c=2$, which is not a wild case. So in the wild case we have a strict equality $b + c - bc < 0$, and therefore $Q(1,1)<0$, $Q(b+1,bc-c+1)<0$.

(2) A direct computation shows $Q(2,1)=4-b<0$, $Q(b+2,3)=Q(b+2,b-1)=(2b+1)(4-b)<0$.
\end{proof}

\begin{proposition}\label{prop:w(1;r)} Using notation in Definition \ref{definition:alpha}, for any $(a_1,a_2) \in \ZZ^2$, and any $k \geq 0$, we have:
$$w(1;k) (a_1, a_2) =\left\{ \begin{array}{ll} 
(a(k-1), a(k)), & \text{ if $k$ is odd};\\   
(a(k), a(k-1)), & \text{ if $k$ is even}.\\
\end{array}     \right.
$$
As a consequence, for $p$ defined in \eqref{eq:p-explicit},
$$
w(1;k) (p) =\left\{ \begin{array}{ll} 
x[\beta(k-1), \beta(k)]+x[\gamma(k-1), \gamma(k)]-x[\alpha(k-1), \alpha(k)], & \text{ if $k$ is odd};\\   
x[\beta(k), \beta(k-1)]+x[\gamma(k), \gamma(k-1)]-x[\alpha(k), \alpha(k-1)], & \text{ if $k$ is even}.\\
\end{array}     \right. \\
$$
\end{proposition}
\begin{proof}
Straightforward induction on $k$.
\end{proof}

In order to treat $w(1;k)(p)$ uniformly, we denote by $\myx[a(k),a(k-1)]$ the greedy element in $\mathcal{A}(r_{k-1},r_k)$ pointed at $(a(k),a(k-1))$ for every $k$. For convenience, we extend the sequences $(\alpha(k))$, $(\beta(k))$, $(\gamma(k))$ to all $k<-1$ using the relation  \eqref{a}.

\begin{proof}[Proof of Lemma \ref{lem:key-lemma}]
We only show that $w(1;r)(p)$ is positive at $\{x_1,x_2\}$ (recall that $p$ is defined in \eqref{eq:p-explicit}), since the treatment of $w(2;r)(p)$ is completely similar and will be left to the reader.

Thanks to Proposition \ref{prop:w(1;r)}, it suffices to prove that the Laurent polynomial
$$p_k=\myx[\beta(k), \beta(k-1)]+\myx[\gamma(k), \gamma(k-1)]-\myx[\alpha(k), \alpha(k-1)]\in\mathcal{A}(r_{k-1},r_k)$$ 
is positive at $\{x_1,x_2\}$ for every $k\ge 0$. Since $p_k$ is the sum of the following two Laurent polynomials
\begin{equation}\label{eq2}
\myx[\gamma(k),\gamma(k-1)]-x_1^{\alpha(k-2)}x_2^{-\alpha(k-1)},
\end{equation}
\begin{equation}\label{eq1}
\myx[\beta(k),\beta(k-1)]+x_1^{\alpha(k-2)}x_2^{-\alpha(k-1)}-\myx[\alpha(k),\alpha(k-1)],
\end{equation}
it suffices to show that \eqref{eq2} and \eqref{eq1} are positive at $\{x_1,x_2\}$. They are proved separately in Lemma \ref{lemma:eq2} and \ref{lemma:eq1}.
\end{proof}

\begin{lemma}\label{lemma:eq2} For every $k\ge0$, \eqref{eq2} is positive at $\{x_1,x_2\}$.
\end{lemma}
\begin{lemma}\label{lemma:eq1} For every $k\ge0$, \eqref{eq1} is positive at $\{x_1,x_2\}$.
\end{lemma}

In order to prove the above two lemmas, we claim some identities among sequences $(\alpha(k)),(\beta(k))$ and $(\gamma(k))$.
\begin{lemma}\label{relation A B}
For every nonnegative integer $k$, we have $\alpha(k)<\beta(k)$. 
For every integer $k$, 
\begin{align}
\label{4.6:eq1} 2\alpha(k-2) + \alpha(k)=\beta(k-2), \\
\label{4.6:eq2}\alpha(k-1) + r_{k} \alpha(k) =\beta(k-1),\\
\label{4.6:eq3}2r_{k}\alpha(k) - \alpha(k-1)=\gamma(k+1),\\
\label{4.6:eq4}\alpha(k) + r_k \alpha(k-1) =\gamma(k),
\end{align}

and
\begin{equation}\label{4.6:eq5}
\aligned
&\alpha(k-1)\beta({k})-\alpha({k})\beta({k-1}) =\alpha(k)\gamma(k-1)-\alpha(k-1)\gamma(k) \\
&\quad\quad\quad=r_{k-1}\alpha(k-1)\alpha(k+1)-r_{k}\alpha(k)^2\\
&\quad\quad\quad =bc\alpha(k-1)\alpha(k)-r_{k-1}\alpha(k-1)^2-r_{k}\alpha(k)^2 = \delta(b,c)>0
\endaligned
\end{equation}
where 
$$\delta(b,c)= \left\{ \begin{array}{ll}bc-b-c & \text{ if } \min(b,c)>1; \\  
 b-4& \text{ if } c=1.    \end{array}     \right.$$
\end{lemma}
\begin{proof}
All identities can be easily proved by induction. To prove $\alpha(k)<\beta(k)$, it suffices to show that $(\beta(0)-\alpha(0),\beta(-1)-\alpha(-1))$ is a positive imaginary root. This is true because in the case $\min(b,c)>1$ we have $Q((c-1)b,c)=bc(b+c-bc)<0$ and in the case $c=1, b>4$, we have 
$Q(b,2)=b(4-b)<0$. 
\end{proof}

We denote by $\myPS[\alpha(k),\alpha(k-1)]$ the pointed support of $\myx[\alpha(k),\alpha(k-1)]\in\mathcal{A}(r_{k-1},r_k)$. Define the map $\varphi_{\alpha(k),\alpha(k-1)}$ by sending $(p,q)$ to $(-\alpha(k)+r_{k-1}p,-\alpha(k-1)+r_kq)$. Similar notation applies with $\alpha$ being replaced by $\beta$ and $\gamma$. Now we are ready to prove Lemma \ref{lemma:eq2}.

\begin{proof}[Proof of Lemma \ref{lemma:eq2}]
The positivity of \eqref{eq2} is equivalent to saying that the support of $\myx[\gamma(k),\gamma(k-1)]$ contains the point $(\alpha(k-2),-\alpha(k-1))$, which follows immediately from Proposition \ref{prop:support} and Lemma \ref{relation A B}. Indeed, 
we need to show that $(p,q):=\varphi^{-1}(\alpha(k-2),-\alpha(k-1))$ lies in the region $\myPS[\gamma(k),\gamma(k-1)]$. 
It is easy to check that 
$$p=(\alpha(k-2)+\gamma(k))/r_{k-1}\stackrel{\eqref{4.6:eq3}}{=}2\alpha(k-1)>\gamma(k)/r_{k-1},$$ 
$$q=(\gamma(k-1)-\alpha(k-1))/r_{k}\stackrel{\eqref{4.6:eq4}}{=}\alpha(k-2)>0.$$  So we only need to show that $(p,q)$ is below the line passing through $(\gamma(k)/r_{k-1},\gamma(k-1)/r_k)$ and $(\gamma(k-1),0)$, which is equivalent to the statement that the three points $(\gamma(k-1),0)$, $(\gamma(k)/r_{k-1},\gamma(k-1)/r_k)$ and $(p,q)$ are in counter-clockwise order. Therefore it follows from
\small
$$
\aligned
&
\begin{vmatrix}1&\gamma(k-1)&0\\1&\gamma(k)/r_{k-1}& \gamma(k-1)/r_k\\1&p&q\\\end{vmatrix}
=\frac{1}{r_{k-1}r_k}\begin{vmatrix}1&r_{k-1}\gamma(k-1)&0\\1&\gamma(k)& \gamma(k-1)\\1&\alpha(k-2)+\gamma(k)&\gamma(k-1)-\alpha(k-1)\\\end{vmatrix}\\
&
=\frac{1}{bc}\begin{vmatrix}0&\gamma(k-2)&-\gamma(k-1)\\1&\gamma(k)& \gamma(k-1)\\0&\alpha(k-2)&-\alpha(k-1)\\\end{vmatrix}
=\frac{1}{bc}\begin{vmatrix}\gamma(k-2)&\gamma(k-1)\\
\alpha(k-2)&\alpha(k-1)\\\end{vmatrix}\stackrel{\eqref{4.6:eq5}}{=}\frac{\delta(b,c)}{bc}>0.
\endaligned
$$
\normalsize
\end{proof}

Lemma \ref{lemma:eq1} follows from the following lemma.
\begin{lemma}\label{mu}
Let $u_i$, $v_i$ be edges in $\mathcal{D}^{\alpha(k)\times \alpha(k-1)}$, and $u'_i$, $v'_i$ be edges in $\mathcal{D}^{\beta(k)\times \beta(k-1)}$. The map 
$$\mu: 
\Bigg{\{}
\begin{array}{l}
\mbox{Compatible pairs }\\
 \mbox{in $\mathcal{D}^{\alpha(k)\times \alpha(k-1)}$}\\
 \end{array} 
 \Bigg{\}}\setminus \Big{\{}(\emptyset,\mathcal{D}^{\alpha(k)\times \alpha(k-1)}_2)\Big{\}}
 \to
\Bigg{\{}
\begin{array}{l}
\mbox{Compatible pairs }\\
 \mbox{in $\mathcal{D}^{\beta(k)\times \beta(k-1)}$}\\
 \end{array} 
 \Bigg{\}} 
 $$
defined by $\mu(S_1,S_2)=(S'_1,S'_2)$ where
$$\aligned
&S'_1=\{u'_1,\dots,u'_{\alpha(k)}\}\cup\{u'_{i+\alpha(k)}| u_i\in S_1\},\\
&S'_2=\{v'_{i+\alpha(k-1)}|v_i\in S_2\}\cup\{v'_{2\alpha(k-1)+1},\dots,v'_{\beta(k-1)}\}.\\
\endaligned
$$
is a well-defined injective map satisfying
$$\aligned
&|S'_1|=|S_1|+(\beta(k-1)-\alpha(k-1))/r_{k}=|S_1|+\alpha(k),\\
&|S'_2|=|S_2|+(\beta(k)-\alpha(k))/r_{k-1}=|S_2|+\alpha(k+1).\big.
\endaligned$$
\end{lemma}
Indeed, assume Lemma \ref{mu} is true. Since the pair $(S_1,S_2)$ contributes a Laurent monomial $x_1^{-\alpha(k)+r_{k-1}p}x_2^{-\alpha(k-1)+r_{k}q}$ to $\myx[\alpha(k),\alpha(k-1)]$, while $(S_1',S_2')$ contributes the same Laurent monomial to  $\myx[\beta(k),\beta(k-1)]$ because
 $$
 x_1^{-\beta(k)+r_{k-1}(p+\alpha(k+1))}x_2^{-\beta(k-1)+r_{k}(q+\alpha(k))}
 \stackrel{\eqref{4.6:eq2}}{=} x_1^{-\alpha(k)+r_{k-1}p}x_2^{-\alpha(k-1)+r_{k}q},
 $$ 
every term in $\myx[\alpha(k),\alpha(k-1)]-x_1^{\alpha(k-2)}x_2^{-\alpha(k-1)}$ is cancelled out by a term in $\myx[\beta(k),\beta(k-1)]$, which implies Lemma \ref{lemma:eq1}.
\medskip

The rest of the paper is devoted to the proof of  Lemma \ref{mu}. For $P\subseteq\mathbb{Z}^2$ and $(d_1,d_2)\in\mathbb{Z}^2$, we denote
$$P-(d_1,d_2)=\{(i-d_1,j-d_2)|(i,j)\in P\}.$$

\begin{lemma}\label{support comparison preparation}  Let $k\geq 0$. Then
$$
\myPS[\alpha(k),\alpha(k-1)])\setminus \{(\alpha(k-1),0)\} \subseteq \myPS[\beta(k),\beta(k-1)]-(\alpha(k+1),\alpha(k))
$$
\end{lemma}
\begin{proof}
Proposition \ref{prop:support} asserts that $\myPS[\alpha(k),\alpha(k-1)]$ is the set of lattice points in the region bounded by the broken line
$$O=(0,0), \,\,K=(\alpha(k-1),0), \,\,L=(\alpha(k)/r_{k-1}, \alpha(k-1)/r_k), \,\, M=(0,\alpha(k)), \,\, (0,0),$$ 
$\myPS[\beta(k),\beta(k-1)]-(\alpha(k+1),\alpha(k))$ is the set of lattice points  in the region bounded by the broken line $O'K'L'M'O'$ with
$O'=(-\alpha(k+1),-\alpha(k))$, $K'=(-\alpha(k+1)+\beta(k-1),-\alpha(k))\stackrel{\eqref{4.6:eq1}}{=}(2\alpha(k-1),-\alpha(k))$, $L'=(-\alpha(k+1)+\beta(k)/r_{k-1}, -\alpha(k)+\beta(k-1)/r_{k})\stackrel{\eqref{4.6:eq2}}{=}L$, $M'=(-\alpha(k+1),-\alpha(k)+\beta(k))\stackrel{\eqref{4.6:eq2}}{=}(-\alpha(k+1),r_{k+1}\alpha(k+1))$.  Note that $O', K', L', M'$ are in the third, fourth, first, and second quadrant, respectively, as shown in Figure \ref{OKLM}. 

\begin{figure}[h]
\begin{tikzpicture}[scale=.8]
\usetikzlibrary{patterns}
\draw (0,0)--(3,0)--(1,1)--(0,2);
\fill [black!10] (0,0)--(3,0)--(1,1)--(0,2)--(0,0);
\draw (-1,-1)--(4,-1)--(1,1)--(-1,4)--(-1,-1);
\draw (0,0) node[anchor=north east] {\tiny{$O$}};
\draw (3,0) node[anchor=south west] {\tiny{$K$}};
\draw (0,2) node[anchor=east] {\tiny{$M$}};
\draw (1,1) node[anchor=south west] {\tiny{$L=L'$}};
\draw (-1,-1) node[anchor=east] {\tiny{$O'$}};
\draw (4,-1) node[anchor=west] {\tiny{$K'$}};
\draw (-1,4) node[anchor=east] {\tiny{$M'$}};
\draw[->] (-3,0) -- (5,0);
\draw[->] (0,-2) -- (0,5);
\draw[dashed](4,-1)--(3,0);
\end{tikzpicture}
  \caption{}
  \label{OKLM}
\end{figure}
 
First, note that $O$ obviously lies strictly inside the region bounded by $O'K'L'M'O'$. 

Next, we claim that $M$ also lies strictly inside the same region.  Indeed, it suffices to show that $L'M'M$ is anti-clockwise oriented, or equivalently, the following determinant
\small
$$
\begin{vmatrix}1&\alpha(k)/r_{k-1}&\alpha(k-1)/r_k\\1&-\alpha(k+1)&r_{k+1}\alpha_{k+1}\\1&0&\alpha(k)\\\end{vmatrix}
\stackrel{\eqref{a}}{=}\frac{1}{bc}(r_{k-1}\alpha(k-1)\alpha(k+1)-r_{k}\alpha(k)^2)
\stackrel{\eqref{4.6:eq5}}{=}\frac{\delta(b,c)}{bc}
$$
\normalsize
is positive, which is obviously the case.

Finally, the area of the triangle $KLK'$ is
\small
$$
\frac{1}{2}\begin{vmatrix}1&\alpha(k-1)&0\\1&\alpha(k)/r_{k-1}&\alpha(k-1)/r_k\\1&2\alpha(k-1)&-\alpha(k)\\\end{vmatrix}
\stackrel{\eqref{a}}{=}\frac{1}{2bc}(bc\alpha(k-1)\alpha(k)-r_{k-1}\alpha(k-1)^2-r_{k}\alpha(k)^2)
\stackrel{\eqref{4.6:eq5}}{=}\frac{\delta(b,c)}{2bc}.$$
\normalsize
Thus $KLK'$ is anti-clockwise oriented. Moreover, since the area is less than $1/2$, Pick's theorem asserts that there is no lattice points $P$ other than $K$ and $K'$ that lie inside or on the boundary of the triangle $KLK'$ (otherwise the triangle $KPK'$, which is contained in the triangle $KLK'$, would have area at least $1/2$). Therefore every lattice point other than $K$ in the region bounded by $OKLMO$ must be in the region bounded by $O'K'L'M'O'$.
\end{proof}

\begin{lemma}\label{support comparison}  Let $k,p,q\geq 0$,  $p<\alpha(k-1)$. If the extremal pair $(S_1,S_2)$ of size $(q;p)$ is compatible in $\mathcal{D}^{\alpha(k)\times \alpha(k-1)}$, then the extremal pair $(S'_1,S'_2)$ of size $(q+\alpha(k);p+\alpha(k+1))$ is  compatible in $\mathcal{D}^{\beta(k)\times \beta(k-1)}$.
\end{lemma}
\begin{proof}
It follows immediately from Proposition \ref{prop:support} and Lemma \ref{support comparison preparation}.
\end{proof}

For any horizontal or vertical edge $u$, we denote
$$
\aligned
&E_u=\textrm{the left/lower endpoint of } u,\\
&F_u=\textrm{the right/upper endpoint of } u.\\
\endaligned
$$

To compare compatible pairs in $\mathcal{D}^{\alpha(k)\times \alpha(k-1)}$ with compatible pairs in $\mathcal{D}^{\beta(k)\times\beta(k-1)}$, we need the following crucial Lemma \ref{lemma:similar path} which roughly says that a certain subpath of the Dyck path $\mathcal{D}^{\beta(k)\times \beta(k-1)}$ is almost identical with $\mathcal{D}^{\alpha(k)\times \alpha(k-1)}$.  

We introduce some notation (see Figure \ref{fig:corner}). Let $B=(\alpha(k),\alpha(k-1))$, $C=(2\alpha(k),2\alpha(k-1))$, $B'=E_{u'_{\alpha(k)+1}}$, $C'=F_{v_{2\alpha(k-1)}}$ (so $C'$ is of the same height as $C$), $G'=B+(1,0)=(\alpha(k)+1,\alpha(k-1))$, $H'=C-(0,1)=(2\alpha(k),2\alpha(k-1)-1)$, $T$ be the intersection of the line $CH'$ with the diagonal $\overline{O'P'}$ with $O'=(0,0)$ and $P'=(\beta(k),\beta(k-1))$. 
Let $BG'H'C$ be the path obtained by taking the union of  the edge $BG'$, the subpath $G'H'$, and the edge $H'C$. Moreover, we denote $B$ by $B_k$ (and we use the subscript $k$ in a similar fashion for other letters) if we need to specify the dependence on $k$.

\begin{figure}[h]
\begin{tikzpicture}[scale=.6]
\draw[line width=.5,color=gray!50] (0,0)--(19,0)--(19,8)--(0,8)--(0,0);
\draw[line width=.5,color=black] (0,0)--(3,0)--(3,1)--(5.4,1)--(5.4,2)
                                                 --(8,2)--(8,3)--(10,3)--(10,4)
                                                 --(10,4)--(13,4)--(13,5)--(15,5)--(15,6)
                                                 --(15,6)--(17,6)--(17,7)--(19,7)--(19,8);
\draw[line width=.5,color=gray!80] (0,0)--(19,8) (0,0)--(8.5,4);
\draw[line width=1, dotted](8.5,4)--(8.5,3);
\draw[line width=1, dotted](4.25,2)--(5.4,2);
\filldraw[black] (4.25,1) circle(2pt);\draw(4.25,1) node[anchor=north]{$B'$};
\filldraw[black] (4.25,2) circle(2pt);\draw(4.25,2) node[anchor=south]{$B$};
\filldraw[black] (5.4,2) circle(2pt);\draw(5.4,2) node[anchor=north west]{$G'$};
\filldraw[black] (8.5,4) circle(2pt);\draw(8.5,4) node[anchor=south]{$C$};
\filldraw[black] (8.5,3) circle(2pt);\draw(8.5,3) node[anchor=north ]{$H'$};
\filldraw[black] (10,4) circle(2pt);\draw(10,4) node[anchor=north west]{$C'$};
\filldraw[black] (8.5,3.6) circle(2pt);\draw(8.5,3.5) node[anchor=west]{$T$};
\draw(0,0) node[anchor=east]{$O'$};
\draw(19,8) node[anchor=west]{$P'$};
\end{tikzpicture}
\caption{}
\label{fig:corner}
\end{figure}

\begin{lemma}\label{lemma:similar path}
 (1)
The points $B, C$ are above $\overline{O'P'}$, and $B'$ has coordinates $(\alpha(k),\alpha(k-1)-1)$ and is below $\overline{O'P'}$, for all $k\ge 0$; $G'$ is below $\overline{O'P'}$ for all $k\ge1$; $H'$ is below $\overline{O'P'}$ for $k\ge 1$ in the case $\min(b,c)>1$, and for $k\ge 2$ in the case $c=1$.

(2) The path $BG'H'C$ has the same shape as $\mathcal{D}^{\alpha(k)\times\alpha(k-1)}$ for $k\ge 1$ in the case $\min(b,c)>1$, and for $k\ge 2$ in the case $c\ge 1$.
\end{lemma}
\begin{proof}
(1) It reduces to the following facts which are easy to check:

 $\beta(k)>\delta(b,c)$  for $k\ge 0$ in both cases $\min(b,c)>1$ and $c=1$;

 $\beta(k)>2\delta(b,c)$ for $k\ge 1$ in the case $\min(b,c)>1$, or for $k\ge 2$ in the case $c=1$.

\noindent Indeed, the statement that $B, C$ are above $\overline{O'P'}$ is equivalent to the inequality $\alpha(k-1)/\alpha(k)>\beta(k-1)/\beta(k)$, which is equivalent to $\delta(b,c)>0$; the statement that $B'$ has coordinates $(\alpha(k),\alpha(k-1)-1)$ and is below $\overline{O'P'}$ follows from the inequality $(\alpha(k-1)-1)/\alpha(k)<\beta(k-1)/\beta(k)$, which is equivalent to $\beta(k)>\delta(b,c)$; the statement that $G'$ is below $\overline{O'P'}$ is equivalent to $\alpha(k-1)/(\alpha(k)+1)<\beta(k-1)/\beta(k)$, which is equivalent to $\beta(k-1)>\delta(b,c)$; the statement that $H'$ is below $\overline{O'P'}$ is equivalent to $(2\alpha(k-1)-1)/(2\alpha(k))<\beta(k-1)/\beta(k)$, which is equivalent to $\beta(k)>2\delta(b,c)$.
 
(2) Based on (1), it suffices to show that there is no lattice points that lies strictly inside the triangle $O'_kT_kC_k$ for all $k\ge1$ in case $\min(b,c)>1$, and all $k\ge2$ in case $c=1$. Aiming at the contradiction, assume that $(i_k,j_k)$ is such a point. In other words,
$\beta(k-1)/\beta(k)<j_k/i_k<\alpha(k-1)/\alpha(k)$, and $0<i_k<2\alpha(k)$. If $k>0$ in case $\min(b,c)>1$, or $k>1$ in case $c=1$,
we would have that the lattice point $(i_{k-1},j_{k-1})=(j_k,r_kj_k-i_k)$ is strictly inside the triangle $O'_{k-1}T_{k-1}C_{k-1}$. Indeed, 
 $$0<i_{k-1}=j_k<i_k\alpha(k-1)/\alpha(k)<2\alpha(k)\alpha(k-1)/\alpha(k)=2\alpha(k-1),$$
 and
 $$r_k-\beta(k)/\beta(k-1)<r_k-i_k/j_k<r_k-\alpha(k)/\alpha(k-1)$$
 implies
 $$\beta(k-2)/\beta(k-1)<j_{k-1}/i_{k-1}<\alpha(k-2)/\alpha(k-1).$$
By descending $k$, we would conclude that:

(i) case $\min(b,c)>1$: there is a lattice point strictly inside the triangle $O'_0T_0C_0$ where $O'_0=(0,0)$, $T_0=(2, 2\beta(0)/\beta(1))$, $C_0=(2,2)$,   which is absurd;

(ii) case $c=1$: there is a lattice point strictly inside the triangle $O'_1T_1C_1$ where $O'_1=(0,0)$, $T_1=(2,2(b+2)/(b-1))$, $C_1=(2,4)$, which is also absurd.
\end{proof}

\begin{lemma}\label{easy lemma}
Given subpath $EF$ of $\mathcal{D}^{\alpha_k\times\alpha_{k-1}}$ or $\mathcal{D}^{\beta_k\times\beta_{k-1}}$ and a pair $(S_1,S_2)$, we define 
$$
\aligned
&f_{S_2}(EF)=r_{k-1}|(EF)_2\cap S_2|-|(EF)_1|,\\
&g_{S_1}(EF)=r_{k}|(EF)_1\cap S_1|-|(EF)_2|.\\
\endaligned
$$
Assuming $u\in S_1$ and $v\in S_2$, the following are equivalent: 

$\bullet$ there exists $A\in (E_uF_v)^\circ$ such that \eqref{0407df:comp} holds, in other words, $f_{S_2}(AF_v)=0$ or $g_{S_1}(E_uA)=0$;

$\bullet$ there exists $A\in (E_uF_v)^\circ$ such that  $f_{S_2}(AF_v)\le 0$ or $g_{S_1}(E_uA)\le 0$;
\end{lemma}
\begin{proof}
Straightforward.
\end{proof}
\begin{proof}[Proof of Lemma \ref{mu}]
It is easy to check that $|S'_1|=|S_1|+\alpha(k)$ and $|S'_2|=|S_2|+\alpha(k+1)$, and $\mu$ is obviously injective if it is well-defined. Thus we are left to show that $(S'_1,S'_2)$ is compatible. 

If $k=0$, or $k=1$ for the case $c=1$, then either $\alpha(k)$ or $\alpha(k-1)$ is equal to 1, thus any compatible pair $(S_1,S_2)$ in $\mathcal{D}^{\alpha(k)\times \alpha(k-1)}$ satisfies either $S_1=\emptyset$ or $S_2=\emptyset$. We immediately conclude that $(S'_1,S'_2)$ are compatible using Lemma \ref{support comparison}.

For the rest of the proof we assume $k\ge 1$ in case $\min(b,c)>1$ or $k\ge 2$ in case $c=1$. Then Lemma \ref{lemma:similar path} (2) can apply and we shall use the notation from there.

We will prove by contradiction by assuming that $(S'_1,S'_2)$ is not compatible. Then by Lemma \ref{easy lemma}, there exist $u'\in S'_1$ and $v'\in S'_2$ such that $f_{S'_2}(AF_{v'})> 0$ and $g_{S'_1}(E_{u'}A)>0$ for every $A\in (E_{u'}F_{v'})^\circ$.

There are four cases to be considered. 

\medskip

\noindent\textbf{Case 1}: both $u',v'$ are on the subpath $B'C'$. Thus we can denote $u'=u'_{i+\alpha(k)}$, $v'=v'_{j+\alpha(k-1)}$ for some $1\le i\le \alpha(k)+\lceil 2\delta(b,c)/\beta(k-1)\rceil$, $0\le j\le \alpha(k-1)$.  
\smallskip

\noindent\textbf{Subcase 1-1}: $u'$ precedes $v'$. Since $(S_1,S_2)$ is compatible, 
without loss of generality we assume the existence of $A\in (E_{u_i}F_{v_j})^\circ$ such that 
$$f_{S_2}(AF_{v_j})=0.$$
Identify $\mathcal{D}^{\alpha(k)\times \alpha(k-1)}$ with $BG'H'C$ thanks to Lemma \ref{lemma:similar path} (2). 
 Since 
$|(AF_{v'})_1|\ge|(AF_{v_j})_1|$ and $|(AF_{v'})_2\cap S'_2|=|(AF_{v_j})_2\cap S_2|$, we get a contradiction
$$f_{S'_2}(AF_{v'})\le f_{S_2}(AF_{v_j})=0.$$

\noindent\textbf{Subcase 1-2}: $u'$ is after $v'$. Then take $A=P'\in (E_{u'}F_{v'})^\circ$,
$$\aligned
&g_{S'_1}(E_{u'}A)=r_k|(E_{u'}P')_1\cap S'_1|-|(E_{u'}P')_2|\\
&\le r_{k}(\alpha(k)-i+1)-\bigg(\beta(k-1)-\Big\lfloor \frac{(\alpha(k)+i-1)\beta(k-1)}{\beta(k)}\Big\rfloor\bigg)\\
&\le r_{k}(\alpha(k)-i+1)-\beta(k-1)+\frac{(\alpha(k)+i-1)\beta(k-1)}{\beta(k)}\\
&=\frac{(i-1)\big(\beta(k-1)-r_k\beta(k)\big)-\beta(k)\big(\beta(k-1)-r_k\alpha(k)\big)+\alpha(k)\beta(k-1)}{\beta(k)}
\endaligned
$$
Since $\beta(k-1)-r_k\beta(k)\stackrel{\eqref{a}}{=}-\beta(k+1)<0$, and $\beta(k-1)-r_k\alpha(k)\stackrel{\eqref{4.6:eq2}}{=}\alpha(k-1)$, we again get a contradiction
$$
g_{S'_1}(E_{u'}A)
\le\frac{-\alpha(k-1)\beta(k)+\alpha(k)\beta(k-1)}{\beta(k)}\stackrel{\eqref{4.6:eq5}}{=}-\frac{\delta(b,c)}{\beta(k)}<0.$$
\medskip

\noindent\textbf{Case 2}: $u'$ is on $B'C'$ but $v'$ is not. Assume $u'=u'_{i+\alpha(k)}, v'=v'_j$ where $1\le i\le \alpha(k), 2\alpha(k-1)+1\le j\le \beta(k-1)$. 
Then we consider the extremal pair $(\tilde{S}'_1,\tilde{S}'_2)$ on $\mathcal{D}^{\beta(k)\times\beta(k-1)}$ of size 
$$(|(F_{u'}C')_1\cap S'_1|+i+\alpha(k);|(F_{u'}C')_2\cap S'_2|+\beta(k-1)-2\alpha(k-1)).$$ (In other words, $(\tilde{S}'_1,\tilde{S}'_2)$ is obtained from $(S'_1,S'_2)$ by moving forward all edges in $(F_{u'}C')_1\cap S'_1$ so that they are immediately after $F_{u'}$, moving backward all edges in $(F_{u'}C')_2\cap S'_2$ so that they are immediately preceding $C'$, and adding all horizontal edges on $B'E_{u'}$ to $S'_1$ and removing all vertical edges on $B'E_{u'}$ from $\tilde{S}'_2$.)

We claim that $(\tilde{S}'_1,\tilde{S}'_2)$ is not compatible. Indeed, for any $A\in (E_{u'}F_{v'})^\circ$, we have
$$|(E_{u'}A)_1\cap \tilde{S}'_1|\ge |(E_{u'}A)_1\cap S'_1|,\quad
|(AF_{v'})_2\cap \tilde{S}'_2|\ge |(AF_{v'})_2\cap S'_2|,$$
thus
$$
g_{\tilde{S}'_1}(E_{u'}A)\ge g_{S'_1}(E_{u'}A)>0,\quad f_{\tilde{S}'_2}(AF_{v'})\ge f_{S'_2}(AF_{v'})>0.
$$

Similarly, we consider the extremal pair $(\tilde{S}_1,\tilde{S}_2)$ on $\mathcal{D}^{\alpha(k)\times\alpha(k-1)}$ of size 
$$(|(F_{u_i}C)_1\cap S_1|+i;|(F_{u_i}C)_2\cap S_2|).$$
(In other words, $(\tilde{S}_1,\tilde{S}_2)$ is obtained from $(S_1,S_2)$ by moving forward all horizontal edges after $u_i$ in $S_1$ so that they become immediately after $u_i$, moving backward all horizontal edges after $u_i$ in $S_2$ to the northeast corner of the Dyck path, and adding all horizontal edges in front of $u_i$ to $S_1$ and removing all vertical edges in front of $u_i$ in $S_2$.)

We claim that $(\tilde{S}_1,\tilde{S}_2)$ is compatible. To see this,
 let 
$$\aligned
&U_1=(BF_{u_i})_1\cap S_1,\quad U_2=(F_{u_i}C)_1\cap S_1=S_1\setminus U_1,\quad W_1=\{u_1,\dots,u_i\}\cup U_2,\\
&V_1=(BF_{u_i})_2\cap S_2,\quad V_2=(F_{u_i}C)_2\cap S_2=S_2\setminus V_1.\\
\endaligned
$$ 
We shall show that $(W_1,V_2)$ is compatible, i.e., every $u_j\in W_1$ and $v_\ell\in V_2$ are separated by $(W_1,V_2)$.
First, consider the case $j\le i$. Since  $(S_1,S_2)$ is compatible, there exists $A\in(E_{u_i}F_{v_\ell})^\circ$ with either $f_{S_2}(AF_{v_\ell})=0$ or $g_{S_1}(E_{u_i}A)=0$. On the other hand,
$g_{S_1}(E_{u_i}A)\ge g_{S'_1}(E_{u'}A)>0$.  So $f_{V_2}(AF_{v_\ell})=f_{S_2}(AF_{v_\ell})=0$.
Next, consider the case $j>i$. Since $(W_1,V_2)$ and $(S_1,S_2)$ coincides on the subpath $E_{u_j}F_{v_\ell}$, a lattice point $A$ that satisfies \eqref{0407df:comp} for $(S_1,S_2)$ also satisfies \eqref{0407df:comp} for $(W_1,V_2)$.
Since $|W_1|=|\tilde{S}_1|$, $|V_2|=\tilde{S}_2$, Proposition \ref{prop:support} implies that $(\tilde{S}_1,\tilde{S}_2)$ is compatible.

Thus we have a compatible extremal pair $(\tilde{S}_1,\tilde{S}_2)$ and an incompatible extremal pair $(\tilde{S}'_1,\tilde{S}'_2)$ which contradict Lemma  \ref{support comparison}.
\medskip

\noindent\textbf{Case 3}: $v'$ is on $B'C'$ but $u'$ is not. The proof of this case is similar to Case 2.
\medskip

\noindent\textbf{Case 4}: neither $u'$ nor $v'$ is on $B'C'$. We let $(\tilde{S}'_1,\tilde{S}_2)$ be the extremal pair of size $(|S'_1|;|S'_2|)$ (which is obtained from $(S'_1,S'_2)$ by moving all horizontal edges in $B'C'\cap S'_1$ forward so that they are immediately after $B'$, and move all vertical edges in $B'C'\cap S'_2$ backward so that they are immediately in front of $C'$). Then $(\tilde{S}'_1,\tilde{S}'_2)$ is not compatible because for any $A\in(E_{u'}F_{v'})^\circ$,  
$$
\aligned
&f_{\tilde{S'_2}}(AF_{v'})\ge f_{S'_2}(AF_{v'})>0,\\
&g_{\tilde{S'_1}}(E_{u'}A)\ge g_{S'_1}(E_{u'}A)>0.\\
\endaligned$$
Similarly define $(\tilde{S}_1,\tilde{S}_2)$ to be the extremal pair of size $(|S_1|;|S_2|)$.
Then by Proposition \ref{prop:support}, $(\tilde{S}_1,\tilde{S}_2)$ is compatible. This again contradicts Lemma \ref{support comparison}.
\medskip

Since in all the four cases we get contradictions, $(S'_1,S'_2)$ must be compatible. This completes the proof of Lemma \ref{mu}.
\end{proof}

\section*{Acknowledgments} 
We are grateful to
Dylan Rupel  for valuable discussions and  many useful comments. 
We thank Gregg Musiker and David Speyer for their interest in this problem.

\end{document}